\documentclass{amsart}
\usepackage{amssymb,amscd, pb-diagram,listings}
\title[Mod 3 generators]
{Mod $3$ Chern classes and generators}
\author{Masaki Kameko}
\address{Department of Mathematical Sciences,
Shibaura Institute of Technology,
307 Minuma-ku Fukasaku, Saitama-City 337-8570, Japan}
\email{kameko@shibaura-it.ac.jp}
\thanks{The author is partially supported 
by the Japan Society for the Promotion of Science, 
Grant-in-Aid for Scientific Research (C) 25400097.}

\newtheorem{theorem}{Theorem}[section]
\newtheorem{proposition}[theorem]{Proposition}
\newtheorem{lemma}[theorem]{Lemma}
\newtheorem{corollary}[theorem]{Corollary}

\theoremstyle{definition}

\newcommand{\spin}{\mathrm{Spin}}

\begin{document}
\maketitle

\begin{abstract}
We show the non-triviality of the mod $3$ Chern class of degree $324$ of the adjoint representation of the exceptional Lie group $E_8$.
\end{abstract}


\section{Introduction} \label{section:1}


Let $p$ be a prime number. In the study of mod $p$ cohomology of the classifying space of 
a simply connected, simple, compact connected Lie group $G$, 
Stiefel-Whitney classes and Chern classes play an important role. For example, the mod $2$ cohomology of 
the classifying space of the exceptional Lie group $E_6$ 
is generated by two generators of degree $4$ and of degree $32$ 
as an algebra over the mod $2$ Steenrod algebra and Toda pointed out that the generator of 
degree $32$ could be given as the Chern class of the irreducible representation $\rho_6:E_6 \to SU(27)$
 in \cite{toda-1973}.
Mimura and Nishimoto \cite{mimura-nishimoto-2007}, Kono \cite{kono-2005} and the 
author \cite{kameko-2012} proved that Stiefel-Whitney classes 
$w_{16}(\rho_4)$, $w_{128}(\rho_8)$ and 
Chern classes $c_{16}(\rho_6)$, $c_{32}(\rho_7)$ are algebra generators 
of the mod $2$ cohomology of the classifying space $BG$ for $G=F_4, E_8, E_6, E_7$, 
where $\rho_4$,  $\rho_8$  are real  irreducible representations of dimension 
$26$, $248$, $\rho_6$, $\rho_7$ are complex irreducible representations of dimension $27$, $56$,
respectively.
For $G=F_4, E_6, E_7$, the mod $2$ cohomology of the classifying space is generated 
by two elements, that is, one is the element of degree $4$ and the other is 
$w_{16}(\rho_4)$, $c_{16}(\rho_6)$, $c_{32}(\rho_7)$, respectively.
In the case $G=E_8$ and $p=2, 3$, the mod $p$ cohomology of the classifying space is not yet computed.
Since the non-triviality of the Stiefel-Whitney class $w_{128}(\rho_8)$ tells us that the differentials in the 
spectral sequence vanishes on the corresponding element, we expect that 
it not only gives us  a nice description for the generator
 but also helps us in the computation of the mod $2$ cohomology of $BE_8$. 
 
 
This paper is the sequel of \cite{kameko-2012} in the sense that 
we consider the mod $3$ analogue of the above results. 
In particular, we prove  the non-triviality of the mod $3$ Chern class $c_{162}(\rho_8)$ of degree $324$. 
For  an odd prime number $p$ and for a  simply-connected, simple, compact connected Lie group, 
the Rothenberg-Steenrod spectral sequence 
collapses at the $E_2$-level and so 
at least additively  the mod $p$ cohomology is isomorphic to the cotorsion product 
of  the mod $p$ cohomology of $G$ except for the case $p=3$, $G=E_8$.
 In \cite{kameko-mimura-2007},
we proved that there exists an algebra generator of degree greater than or equal to $324$ in the mod $3$ cohomology ring of $BE_8$.  On the other hand,  in \cite{mimura-sambe-1980}, Mimura and Sambe proved that the $E_2$-term of the Rothenberg-Steenrod spectral sequence 
is generated as an algebra by elements of degree less than or equal to $168$.
Hence the spectral sequence must not collapse at the $E_2$-level. 
We expect that, in the mod $3$ cohomology,  the mod $3$ Chern class $c_{162}(\rho_8)$ plays an important  role  
similar to that of the  Stiefel-Whitney class $w_{128}(\rho_8)$ in the mod $2$ cohomology.


Now, we state our main theorem. Let $T$ be a fixed maximal torus of the exceptional Lie group $F_4$. We choose a maximal non-toral elementary abelian $3$-subgroup $A$ of $F_4$
so that $T\cap A$ is nontrivial. We refer the reader to the paper of Andersen, Grodal, M{\o}ller and Viruel \cite[Section 8]{andersen-2008} for the details of non-toral elementary abelian $p$-subgroups of exceptional Lie
 groups and their Weyl groups.
Let $\mu$ be a subgroup of $T\cap A$ of order $3$. The group  $\mu$ is the cyclic group of order $3$. 
We consider the following diagram of inclusion maps.
\[
\begin{diagram}
\node{T}\arrow{e} \node{F_4}\arrow{e} \node{G} \\
\node{\mu} \arrow{n} \arrow{e} \node{A.}\arrow{n} 
\end{diagram}
\]
We denote by $\iota:\mu \to G$ the inclusion map of $\mu$ to $G=F_4, E_6, E_7, E_8$.
The mod $3$ cohomology $H^{*}(B\mu;\mathbb{Z}/3)$ of the classifying space $B\mu$ is isomorphic to
\[
\mathbb{Z}/3[u_2]\otimes \Lambda(u_1),
\]
where $u_2$ is the image of the mod $3$ Bockstein homomorphism of a generator $u_1$ of $H^{1}(B\mu;\mathbb{Z}/3)=\mathbb{Z}/3$.
From now on, we consider complex representations only 
and we denote  complexifications of 
real representations $\rho_4$, $\rho_8$ by  the same symbols
$\rho_4$, $\rho_8$, respectively.


\begin{theorem}
\label{theorem:1.1}
The total Chern classes $c(\iota^*(\rho_i))$ of the above induced representations  $\iota^*(\rho_i)$, where $i=4, 6, 7,8$,  are as follows:
\begin{eqnarray*}
c(\iota^{*}(\rho_4))&=&1-u_2^{18}, 
\\
c(\iota^{*}(\rho_6))&=&1-u_2^{18},
\\
c(\iota^{*}(\rho_7))&=& (1-u_2^{18})^2=1+u_2^{18}+u_2^{36},
\\
c(\iota^{*}(\rho_8))&=&(1-u_2^{18})^{9}=1-u_2^{162}.
\end{eqnarray*}
\end{theorem}

As a corollary of this theorem, using Lemma~\ref{lemma:3.1},  we have the following:

\begin{corollary}
\label{corollary:1.2}
The Chern classes 
$c_{18}(\rho_4)$, $c_{18}(\rho_6)$, $c_{18}(\rho_7)$, $c_{162}(\rho_8)$
are nontrivial in $H^{*}(BF_4;\mathbb{Z}/3)$, 
$H^{*}(BE_6;\mathbb{Z}/3)$,
$H^{*}(BE_7;\mathbb{Z}/3)$,
$H^{*}(BE_8;\mathbb{Z}/3)$, respectively.
Moreover, the Chern classes 
$c_{18}(\rho_4)$, $c_{18}(\rho_6)$, $c_{18}(\rho_7)$ are indecomposable, so that they are algebra generators.
\end{corollary}

This paper is organized as follows:
In Section~\ref{section:2}, we recall complex representations $\rho_4$, $\rho_6$, $\rho_7$, $\rho_8$
and their restrictions to $\spin(8)$.
In Section~\ref{section:3}, we prove Theorem~\ref{theorem:1.1}.
We end  this paper by showing the non-triviality of the mod $5$ Chern class $c_{100}(\rho_8)$ of $BE_8$
in the appendix.


\section{Complex representations}  \label{section:2}


In this section, we consider complex representations $\rho_4, \rho_6, \rho_7, \rho_8$ 
in Theorem~\ref{theorem:1.1} and the complexification of the adjoint representation $\rho'_4$ of $F_4$ and their restrictions to $\spin(8)$.
For the details of representation rings of Spin groups and cyclic groups,  we refer the reader to standard textbooks on representation theory, e.g. Husemoller's book \cite{husemoller-book-1994} and/or 
the  book of Br\"ocker and T. tom Dieck \cite{brocker-tomdieck-1995}.

First, we recall complex representation ring of $\spin(2n)$.
Let us consider the following pull-back diagram.
\[
\begin{diagram}
\node{\tilde{T}^n} \arrow{e,t}{\tilde{k}_n} \arrow{s,l}{\pi} 
\node{\spin(2n)} \arrow{s,r}{\pi}
\\
\node{T^n} \arrow{e,t}{k_n} \node{SO(2n),}
\end{diagram}
\]
where $SO(2n)$ is the special orthogonal group, $\pi:\spin(2n)\to SO(2n)$ is the universal covering, $T^n$ is the maximal torus of $SO(2n)$ consisting of 
matrices of the form
\[
\left(
\begin{array}{cr|ccc|cr}
\cos \theta_1 & -\sin \theta_1 & & & & \\ 
\sin \theta_1 & \cos \theta_1 & & & & \\   \hline 
&  & & & & \\ 
 &  & & \ddots & & \\  
  &  & & & &  \\
   \hline
  &  & &  & & \cos \theta_n & -\sin \theta_n \\ 
  &  & &  &  & \sin \theta_n & \cos  \theta_n 
  \end{array}  \right), 
  \]
and  $\tilde{T}^n$ is a maximal torus of $\spin(2n)$.
The complex representation ring of
\[
S^1=\left\{ 
\left( \begin{array}{cr} \cos \theta & -\sin \theta \\ \sin \theta & \cos \theta \end{array} 
\right) \right\}
\]
 is $R(S^1)=\mathbb{Z}[z,z^{-1}]$ where $z$ is represented by the canonical complex line bundle.
 Considering $T^n$ as the product of $n$ copies of $S^1$'s,
let $p_i:T^n \to S^1$ be the projection to the $i$-th factor.
We denote by $z_i$ the element 
$p_i^{*}(z)$, $\pi^{*}(p_i^{*}(z))$ in $R(T^n)$, $R(\tilde{T}^n)$, respectively, 
so that $\pi^{*}(z_i)=z_i$.
Then, we have
\begin{eqnarray*}
R(T^n)&=& \mathbb{Z}[z_1, \dots, z_n, (z_1\cdots z_n)^{-1}], \\
R(\tilde{T}^n)&=& \mathbb{Z}[z_1, \dots, z_n, (z_1\cdots z_n)^{-1/2}]
\end{eqnarray*}
and the complex representation ring of $\spin(2n)$ is 
\[
\mathbb{Z}[\lambda_1, \dots, \lambda_{n-1}, \Delta^{+}, \Delta^-]
\]
where
 \begin{eqnarray*}
 \tilde{k}_n^*(\lambda_1) &=& \sum_{i=1}^n (z_i + z_i^{-1}), 
 \\
  \tilde{k}_n^*(\lambda_2) &=& \sum_{1\leq i<j  \leq n} (z_i+z_i^{-1})(z_j+z_j^{-1}), 
  \\
    \tilde{k}_n^*(\Delta^+) &=& \sum_{\varepsilon_1 \cdots \varepsilon_n=1} (z_1^{\varepsilon_1}\cdots z_n^{\varepsilon_n})^{1/2}, 
    \\
     \tilde{k}_n^*(\Delta^-) &=& \sum_{\varepsilon_1 \cdots \varepsilon_n=-1} (z_1^{\varepsilon_1}\cdots z_n^{\varepsilon_n})^{1/2}, 
\end{eqnarray*}
and $\varepsilon_i \in \{ \pm 1\}$. For the sake of notational simplicity, from now on, we write $\Delta$ for $\Delta^++\Delta^-$. 
Let $i:\mu\to S^1$ be the inclusion map. We denote by $z$ the generator $i^*(z)$ of $R(\mu)$.
Then, it is also known that $R(\mu)=\mathbb{Z}[z]/(z^3)$.

Next, we recall complex representations $\rho_4$, $\rho_6$, $\rho_7$, $\rho_8$ of dimension $26$, $27$, $56$, $248$ in Section~\ref{section:1}
and the complexification $\rho_4'$ of the adjoint representation of $F_4$.
We consider the following commutative diagram.
\[
\begin{diagram}
\node{\spin(8)} \arrow{e,t}{i_{8}} \arrow{s,r}{ j_8 } 
\node{\spin(10)} \arrow{e,t}{ i_{10} } \arrow{s,r}{ j_{10} } 
\node{\spin(12)} 
\arrow{e,t}{i_{14} \circ i_{12}} 
\arrow{s,r}{ j_{12} } 
\node{\spin(16)} \arrow{s,r}{j_{16}}
\\
\node{F_4} \arrow{e,t}{i_4} \node{E_6} \arrow{e,t}{i_6} \node{E_7} \arrow{e,t}{i_7} \node{E_8,}
\end{diagram}
\]
where 
\[
i_{2n-2}:\spin(2n-2) \to \spin(2n)
\]
is the obvious inclusion map.
For $\rho_4$, $\rho_4'$, we refer the reader to Yokota's paper \cite{yokota-1967-f4}.  For 
$\rho_6, \rho_7$, we refer the reader to Adams' book \cite[Corollaries 8.3, 8.2]{adams-1996}.
For $E_8$, from the construction of $E_8$ in Adams \cite[Section 7]{adams-1996} and the fact that the adjoint representation of $\spin(2n)$ is the 
second exterior power of the standard representation, we have the following proposition.


\begin{proposition} 
\label{proposition:2.1}
We have
\begin{eqnarray*}
j_8^{*}(\rho_4)&=& 2+\lambda_1+\Delta, \\
j_8^{*}(\rho_4')&=& 4 + \lambda_1+\Delta+\lambda_2, \\
j_{10}^*(\rho_6)&=& 1+\lambda_1+\Delta^+, 
\\
j_{12}^*(\rho_7)&=& 2\lambda_1+\Delta^-,
\\
j_{16}^*(\rho_8)&=& 8+\lambda_2+\Delta^+, 
\end{eqnarray*}
in $R(\spin(8))$, $R(\spin(8))$, $R(\spin(10))$, $R(\spin(12))$, $R(\spin(16))$, respectively.
\end{proposition}

Since the induced homomorphism $i_{2n-2}^*$  maps
$\lambda_1$, $\lambda_2$, $\Delta^+$, $\Delta^-$, $\Delta$
to $2+\lambda_1$, 
$2\lambda_1+\lambda_2$, 
$\Delta$, $\Delta$, $2\Delta$, 
respectively,
we have the following proposition.


\begin{proposition}
\label{proposition:2.2}
For $G=F_4, E_6, E_7, E_8$, let $\mathop{j}: \spin(8) \to G$ be the inclusion map.
In $R(\spin(8))$, we have
\begin{eqnarray*}
j^{*}(\rho_4)&=& 2+ \lambda_1 +\Delta, \\
j^{*}(\rho_6)&=& 3+ \lambda_1 +\Delta, \\
j^*(\rho_7)&=& 8+ 2\lambda_1+2\Delta, \\
j^*(\rho_8)&=& 32+8\lambda_1+8\Delta+\lambda_2.
\end{eqnarray*}
\end{proposition}


\section{Mod $3$ Chern classes} 
\label{section:3}

In this section, we prove Theorem~\ref{theorem:1.1}. We consider the following diagram 
of inclusion maps.
\[
\begin{diagram}
\node{\tilde{T}^4} \arrow{e,t}{\tilde{k}_4} \node{\spin(8)} \arrow{e,t}{j_8} \node{F_4} \\
\node{\mu} \arrow{n,l}{\iota_0} \arrow[2]{e,t}{\iota_1} \node[2]{A.}\arrow{n}
\end{diagram}
\]
The maximal torus $\tilde{T}^4$ of $\spin(8)$  is the maximal torus $T$ of $F_4$ we mentioned in Section 1.
By abuse of notation, we denote both the inclusion map of $\mu$ to $\tilde{T}^4$ 
and its composition with $\tilde{k}_4$ by the same symbol $\iota_0$. 
Let $\sqrt{0}$ be the nilradical of $H^{*}(BA;\mathbb{Z}/3)$ and $H^{*}(B\mu;\mathbb{Z}/3)$, so that we have the induced homomorphism
\[
\iota_1^*:H^{*}(BA;\mathbb{Z}/3)/\sqrt{0} \to H^{*}(B\mu;\mathbb{Z}/3)/\sqrt{0}=\mathbb{Z}/3[u_2].
\]

 
\begin{lemma}
\label{lemma:3.1}
The image of the induced homomorphism
\[
\iota^*:H^{*}(BF_4;\mathbb{Z}/3)\to H^{*}(B\mu;\mathbb{Z}/3)/\sqrt{0}
\]
is in $\mathbb{Z}/3[ u_2^{18}]$, i.e. $\mathop{\mathrm{Im}} \iota^* \subset \mathbb{Z}/3[u_2^{18}]\subset \mathbb{Z}/3[u_2]$.
\end{lemma}

\begin{proof}
It is well-known that the Weyl group $W(A)=N(A)/C(A)$  of $A$ in $F_4$ is isomorphic to the special linear group $SL_3(\mathbb{Z}/3)$. See the paper of Andersen, Grodal, M{\o}ller and Viruel \cite[Section 8]{andersen-2008}.
Moreover, $H^{*}(BA;\mathbb{Z}/3)/\sqrt{0}$ is a polynomial algebra with $3$ variables of degree $2$ 
and $SL_3(\mathbb{Z}/3)$ acts in the usual manner.
The ring of invariants is also a  polynomial algebra 
\[
(H^{*}(BA;\mathbb{Z}/3)/\sqrt{0})^{W(A)}=\mathbb{Z}/3[e_3, c_{3,1}, c_{3,2}].
\]
The invariants $e_3^2=c_{3,0}, c_{3,1}, c_{3,2}$ are known as Dickson invariants and their degrees are 
$52, 48, 36$, 
respectively. 
Moreover, the induced homomorphism $\iota_1^*$ maps $c_{3,0}, c_{3,1}, c_{3,2}$ to $0$, $0$, $u_2^{18}$, respectively. See Wilkerson's paper \cite[Corollary 1.4]{wilkerson-1982} for the details.
Since the induced homomorphism $\iota^*$ factors through
\[
(H^{*}(BA;\mathbb{Z}/3)/\sqrt{0})^{W(A)} \to H^{*}(B\mu;\mathbb{Z}/3)/\sqrt{0} , 
\]
the lemma follows.
\end{proof}

Next, we compute the total Chern class $c(\iota_0^*(\lambda_1 +\Delta))$.


\begin{proposition}
\label{proposition:3.2}
The total Chern class $c(\iota_0^*(\lambda_1 +\Delta))$ is equal to $1-u_2^{18}$.
\end{proposition}

\begin{proof}
Since $\dim (\lambda_1+\Delta)=24$, 
and since $c(\iota_0^*(\lambda_1+\Delta))
=c(\iota^*(\rho_4)) \in \mathbb{Z}/3[u_2^{18}]$ by Lemma~\ref{lemma:3.1}, 
$c(\iota_0^*(\lambda_1+\Delta))$ is equal to $1+\alpha u_2^{18}$ for some $\alpha \in \mathbb{Z}/3$.
On the other hand, 
$\iota_0^*$ maps $z_i$ to $z^{\alpha_i}$ for some $\alpha_i \in \mathbb{Z}/3$ and, since $\iota_0$ is an inclusion map, 
$(\alpha_1, \alpha_2, \alpha_3, \alpha_4) \not= (0,0,0,0)$.
So, 
\[
c(\iota_0^*(\lambda_1))=\prod_{i=1}^{4} (1-\alpha_i^2 u_2^2)
\]
and $\alpha_i \not = 0$ for some $i$. Hence, $c(\iota_0^*(\lambda_1))$ is divisible by $1-u_2^2$.
Therefore, 
\[
c(\iota_0^{*}(\lambda_1+\Delta))=c(\iota_0^*(\lambda_1))c(\iota_0^*(\Delta))
\]
 is also divisible by $1-u_2^2$ and so $\alpha=-1$ in $\mathbb{Z}/3$.
\end{proof}

Next, we compute the total Chern class $c(\iota_0^*(\lambda_2))$.


\begin{proposition}
\label{proposition:3.3}
The total Chern class $c(\iota_0^*(\lambda_2))$ is equal to $1-u_2^{18}$.
\end{proposition}

\begin{proof}
As in the proof of the previous proposition, 
assume that $\iota_0^{*}(z_i)=z^{\alpha_i}$.
Let 
\[
f_{ij}=(1-(\alpha_i+\alpha_j) u_2)(1-(\alpha_i-\alpha_j) u_2)(1-(-\alpha_i+\alpha_j) u_2)(1-(-\alpha_i-\alpha_j) u_2).
\]
Then, \[
c(\iota_0^*(\lambda_2))=\prod_{1\leq i<j\leq 4}  f_{ij}
\]
and
\[
f_{ij}=1-2(\alpha_i^2+\alpha_j^2) u_2^2+(\alpha_i^2-\alpha_j^2)^2u_2^4.
\]
For $(\alpha_i^2, \alpha_j^2)=(1,1)$, we have $f_{ij}=1-u_2^2$.
For $(\alpha_i^2, \alpha_j^2)=(1,0)$ or $(0,1)$, we have $f_{ij}=1-2u_2^2+u_2^4=(1-u_2^2)^2$.
Since $(\alpha_1, \alpha_2, \alpha_3, \alpha_4)\not=(0,0,0,0)$, there exists $(i,j)$ 
such that $(\alpha_i, \alpha_j)\not = (0,0)$. Hence the total Chern class 
$c(\iota_0^*(\lambda_2))$ is not trivial and it is 
divisible by $1-u_2^2$.

Let us consider the total Chern class $c(\iota^*(\rho_4'))$. By Lemma~\ref{lemma:3.1}, it is in $\mathbb{Z}/3[u_2^{18}]$ and by Proposition~\ref{proposition:3.2}, we have
\[
c(\iota^*(\rho_4'))=c(\iota_0^*(\lambda_2)) c(\iota_0^*(\lambda_1+\Delta))=c(\iota_0^*(\lambda_2))(1-u_2^{18}). 
\]
So, $c(\iota_0^*(\lambda_2))$ is also in $\mathbb{Z}/3[u_2^{18}]$. Since $\dim \lambda_2=24$, 
$c(\iota_0^*(\lambda_2))=1+\alpha u_2^{18}$ for some $\alpha\in \mathbb{Z}/3$. Since $c(\iota_0^*(\lambda_2))$
is divisible by $1-u_2^2$, $\alpha=-1$ as in the proof of the previous proposition.
\end{proof}

Finally, we prove Theorem~\ref{theorem:1.1}.

\begin{proof}[Proof of Theorem~\ref{theorem:1.1}]
Using Propositions~\ref{proposition:2.1}, \ref{proposition:2.2} and using Propositions~\ref{proposition:3.2}, \ref{proposition:3.3} above, we have
\[
\begin{array}{rclcl}
c(\iota^*(\rho_4))&=& c(\iota_0^*(\lambda_1+\Delta)) &=& 1-u_2^{18}, \\
c(\iota^*(\rho_6))&=& c(\iota_0^*(\lambda_1+\Delta))&=&1-u_2^{18}, \\
c(\iota^*(\rho_7))&=& c(\iota_0^*(\lambda_1+\Delta))^2&=&( 1-u_2^{18})^2, \\
c(\iota^*(\rho_4))&=& c(\iota_0^*(\lambda_1+\Delta))^8 c(\iota_0^*(\lambda_2))&=& (1-u_2^{18})^9.  \qedhere
\end{array}
\]
\end{proof}


\appendix
\section{Mod $5$ Chern classes} 
\label{section:4}

Let $p$ be an odd prime number. Let $G$ be a  simply-connected, simple, compact connected Lie group.
If the integral homology of $G$ has no $p$-torsion, then the mod $p$ cohomology ring of its classifying space 
is a polynomial algebra and it is well-known. 
See, for example, the book of Mimura and Toda \cite{mimura-toda-1991}.
The integral homology of $G$ has $p$-torsion if and only if $(G, p)$ is one of 
$(F_4, 3), (E_6, 3), (E_7, 3), (E_8,3)$ and $(E_8, 5)$.
We dealt with the cases for $p=3$ in this paper. 
For completeness, in this appendix, we deal with the remaining case, $p=5$, $G=E_8$, that is, 
we prove the non-triviality of the mod $5$ Chern class $c_{100}(\rho_8)$ of the complexification of 
the adjoint representation $\rho_8$ 
of the exceptional Lie group $E_8$.

The mod $5$ analogue of Corollary~\ref{corollary:1.2} is as follows:

\begin{theorem}
\label{theorem:4.1}
The mod $5$ Chern class $c_{100}(\rho_8)$ is non-trivial.
Moreover, the mod $5$ Chern class $c_{100}(\rho_8)$ is indecomposable in $H^{*}(BE_8;\mathbb{Z}/5)$.
\end{theorem}

To prove this theorem, we need the mod $5$ analogue of Lemma~\ref{lemma:3.1}.
As in the case $p=3$, $G=F_4$, there exists a non-toral maximal elementary abelian $5$-subgroup of rank $3$ in 
the exceptional Lie group $E_8$. 
We choose the maximal torus $T$ of $E_8$.
I necessary by replacing $A$ by its conjugate, we may assume that $A\cap T$ is non-trivial.
We choose a subgroup $\mu$ of $A\cap T$ of order $5$. 
Indeed, it is the cyclic group of order $5$. We denote by $\iota:\mu\to E_8$ the inclusion map.
The mod $5$ cohomology of $B \mu$ is 
\[
H^{*}(B\mu;\mathbb{Z}/5)=\mathbb{Z}/5[u_2] \otimes \Lambda(u_1),
\]
where $u_1$ is a generator of $H^{1}(B\mu;\mathbb{Z}/5)=\mathbb{Z}/5$ 
and $u_2$ is its image by the mod $5$ Bockstein homomorphism.
As in the previous section, we denote the nilradical by $\sqrt{0}$
and we denote the inclusion map of $\mu$ to $A$ by $\iota_1:\mu \to A$.


\begin{lemma}
\label{lemma:4.2}
The image of the induced homomorphism 
\[
\iota^*:H^{*}(BE_8;\mathbb{Z}/5) \to H^{*}(B\mu;\mathbb{Z}/5)/\sqrt{0}
\]
is in $\mathbb{Z}/5[u_2^{100}]\subset H^{*}(B\mu;\mathbb{Z}/5)/\sqrt{0}$.
\end{lemma}

\begin{proof}
Since the induced homomorphism $\iota^*$ factors through 
\[
\iota_1^*: (H^{*}(BA;\mathbb{Z}/5)/\sqrt{0})^{W(A)} \to H^{*}(B\mu;\mathbb{Z}/5)/\sqrt{0},
\]
all we need to do is to recall the fact that the Weyl group $W(A)$ of $A$ in $E_8$ is $SL_3(\mathbb{Z}/5)$,
that 
\[
(H^{*}(BA;\mathbb{Z}/5)/\sqrt{0})^{W(A)}=\mathbb{Z}/5[e_3, c_{3,2}, c_{3,1}]
\]
and that 
the above induced homomorphism $\iota_1^*$ maps $e_3, c_{3,1}, c_{3,2}$  to $0, 0, u_2^{100}$, 
respectively
We find these facts in \cite[Section 8]{andersen-2008} and in  \cite[Corollary 1.4]{wilkerson-1982}.
\end{proof}

To compute $\iota^*(\rho_8)$, we need the following commutative diagram similar to the diagram in Section~\ref{section:3}.
However, in this case, 
the map $j_{16}:\spin(16) \to E_8$ is not  injective. 
\[
\begin{diagram}
\node{\tilde{T}^8} \arrow{e,t}{\tilde{k}_8} \node{\spin(16)} \arrow{e,t}{j_{16}} \node{E_8} \\
\node{\mu} \arrow{n,l}{\iota_0}  \arrow[2]{e,t}{\iota_1}\node[2]{A} \arrow{n}
\end{diagram}
\]
We choose the maximal torus $T$ of $E_8$ so that $j_{16}(\tilde{T}^{8})=T$.
Then, since $\tilde{T}^8 \to T$ is a double cover and since $\mu$ is of order $5$, 
there exists a map $\iota_0:\mu \to \tilde{T}^8$ 
such  that the above diagram commutes.

We use the following propositions to prove Theorem~\ref{theorem:4.1}.


\begin{proposition}
\label{proposition:4.3}
The total mod $5$ Chern class of $\iota_0^*(\lambda_2)$ is a product of copies of $1-u_2^2$ and $1+u_2^2$.
Moreover, it is non-trivial.
\end{proposition}
\begin{proof}
Let 
\[
f_{ij}=(1-(\alpha_i+\alpha_j)u_2)(1-(-\alpha_i+\alpha_j)u_2)(1-(\alpha_i-\alpha_j)u_2)(1-(-\alpha_i-\alpha_j)u_2).
\]
Then, we have
\[
c(\iota_0^*(\lambda_2))=\prod_{1\leq i< j\leq 8} f_{ij}
\]
and 
\[
f_{ij}=1-2(\alpha_i^2+\alpha_j^2) u_2^2+(\alpha_i^2-\alpha_j^2)^2 u_2^4.
\]
In $\mathbb{Z}/5$, $\alpha_i^2=0$ or $\pm 1$.
So, 
$f_{ij}=1+u_2^2$ for $(\alpha_i^2, \alpha_j^2)=(1,1)$, 
$f_{ij}=1-u_2^2$ for $(\alpha_i^2, \alpha_j^2)=(-1,-1)$, 
$f_{ij}=(1-u_2^2)^2$ for $(\alpha_i^2, \alpha_j^2)=(1,0), (0,1)$, 
$f_{ij}=(1+u_2^2)^2$ for $(\alpha_i^2, \alpha_j^2)=(-1,0), (0,-1)$, 
$f_{ij}=1$ for $(\alpha_i^2, \alpha_j^2)=(0,0)$.
Since $\mu$ is a non-trivial subgroup of $\tilde{T}^8$,  $\alpha_i$ is non-zero for some $i$. So, the total Chern class 
is not equal to $1$
and so we have the proposition.
\end{proof}


\begin{proposition}
\label{proposition:4.4}
The total mod $5$ Chern class of $\iota_0^*(\Delta^+)$ is also a product of copies of $1-u_2^2$ and $1+u_2^2$.
\end{proposition}
\begin{proof}
Suppose that 
$i_0^*:R(\spin(16)) \to R(\mu)$ maps $(z_1^{\varepsilon_1} \cdots z_8^{\varepsilon_8})^{1/2}$ to 
$z^{\alpha_{\varepsilon_1\dots\varepsilon_8}}$.
Then, it maps
$(z_1^{\varepsilon'_1} \cdots z_8^{\varepsilon'_8})^{1/2}$ to $z^{-\alpha_{\varepsilon_1\dots\varepsilon_8}}$, where
$\varepsilon_i'=-\varepsilon_i$, 
and we have
\[
c(\iota_0^*(\Delta^+))=\prod_{\varepsilon_1=1, \varepsilon_1 \varepsilon_2\cdots \varepsilon_8=1} (1-\alpha_{\varepsilon_1\varepsilon_2\dots\varepsilon_8}^2 u_2^2).
\]
Since $\alpha_{\varepsilon_1\dots \varepsilon_8}^2=0$ or $\pm 1$, we have the desired result.
\end{proof}

Now we complete the proof of Theorem~\ref{theorem:4.1}.

\begin{proof}[Proof of Theorem~\ref{theorem:4.1}]
By Propositions~\ref{proposition:4.3}, \ref{proposition:4.4}, 
the total Chern class $c(\iota^*(\rho_8))$ is a product of copies of $1-u_2^2$ and $1+u_2^2$
and it is non-trivial.
On the other hand, by Lemma~\ref{lemma:4.2}, since $\dim (\lambda_2+\Delta^+)=240$, 
\[
c(\iota^*(\rho_8))=1+\alpha u_2^{100}+\beta u_2^{200}
\]
for some $\alpha, \beta \in \mathbb{Z}/5$ and $(\alpha, \beta) \not=(0,0)$. Since it is divisible by $1-u^2$ or $1+u_2^2$, 
we have $1+\alpha+\beta=0$ in $\mathbb{Z}/5$
and 
\[
c(\iota^*(\rho_8))=1+(-\beta-1) u_2^{100}+\beta u_2^{200}=(1-u_2^{100})(1-\beta u_2^{100}).
\]
Since it is a product of copies of $1-u_2^2$ and $1+u_2^2$, $1+\beta u_2^{100}$ is also divisible by $1-u_2^2$
or $1+u_2^2$ if $\beta \not=0$.
So, $\beta=0$ or $-1$ and we have that $c(\iota^*(\rho_8))$ is equal to $1-u_2^{100}$ or $(1-u_2^{100})^2$.
In particular, $c_{100}(\rho_8)=-u_2^{100}$ or $-2u_2^{100}$ and by Lemma~\ref{lemma:4.2}, it is indecomposable in $H^{*}(BE_8;\mathbb{Z}/5)$.
\end{proof}

\end{document}